\documentclass{amsart}
\usepackage{amsfonts,amssymb,amscd,amsmath,latexsym,amsbsy,mathrsfs,amsxtra}
\usepackage{mathtools,tikz-cd}
\usetikzlibrary{cd}
\usetikzlibrary{arrows}
\usetikzlibrary{matrix}
\usepackage{graphicx,ctable,booktabs}
\usepackage[all]{xy}
\usetikzlibrary{arrows.meta}
\numberwithin{equation}{section}            
\theoremstyle{plain}
\newtheorem{thm}{Theorem}[section]

\newtheorem{prop}[thm]{Proposition}

\newtheorem{defi}[thm]{Definition}
\newtheorem{lem}[thm]{Lemma}
\newtheorem{cor}[thm]{Corollary}
\newtheorem{eg}[thm]{{Example}}

\theoremstyle{remark}
\newtheorem{rema}[thm]{Remark}

\usepackage{enumerate}
\usepackage{comment}

\newcommand{\bc}{{\mathbf{c}}}

\newcommand{\bs}{{\bf{s}}}

\newcommand{\N}{{\mathbb N}}

\newcommand{\field}{{\mathbb K}}

\newcommand{\gfrak}{{\mathfrak g}}

\newcommand{\kfrak}{{\mathfrak k}}

\newcommand{\kow}{{\varDelta}}

\newcommand{\ot}{\otimes}

\newcommand{\Z}{{\mathbb Z}}
\newcommand{\wt}{\widetilde}

\newcommand{\bbt}{\mathbb{T}}
\newcommand{\bbr}{\mathbb{R}}
\newcommand{\bbc}{\mathbb{C}}
\newcommand{\ol}[1]{\overline{#1}}
\newcommand*\pFqskip{8mu}
\catcode`,\active
\newcommand*\pFq{\begingroup
        \catcode`\&\active
        \def ,{\mskip\pFqskip\relax}%
        \dopFq
}
\catcode`\,12
\def\dopFq#1#2#3#4#5{%
        {}_{#1}\phi_{#2}\biggl[\genfrac..{0pt}{}{#3}{#4};#5\biggr]%
        \endgroup
}

\begin{document}
\title[Bivariate $q$-Hermite polynomials and deformed Serre relations]
{Bivariate continuous $q$-Hermite polynomials \\
and deformed quantum Serre relations}
\dedicatory{To Nicol\'as Andruskiewitsch on his 60th birthday, with admiration}
\author[W. Riley Casper]{W. Riley Casper}
\address{
Department of Mathematics \\
California State University \\
Fullerton, CA 92831 \\
U.S.A.
}
\email{wcasper@fullerton.edu}
\author[Stefan Kolb]{Stefan Kolb}
\address{School of Mathematics, Statistics and Physics,
Newcastle University, Newcastle upon Tyne NE1 7RU, United Kingdom}
\email{stefan.kolb@newcastle.ac.uk}
\author[Milen Yakimov]{Milen Yakimov}
\address{
Department of Mathematics \\
Louisiana State University \\
Baton Rouge, LA 70803 \\
U.S.A.
}
\email{yakimov@math.lsu.edu}



\begin{abstract}
We introduce bivariate versions of the continuous $q$-Hermite polynomials. We
obtain algebraic properties for them (generating function, explicit expressions 
in terms of the univariate ones, backward difference equations and recurrence relations) 
and analytic properties (determining the orthogonality measure). 
We find a direct link between bivariate continuous $q$-Hermite polynomials 
and the star product method of \cite{a-KY19p} for quantum symmetric pairs 
to establish deformed quantum Serre relations for quasi-split quantum symmetric pairs of Kac-Moody type. 
We prove that these defining relations are obtained from the usual quantum Serre relations by 
replacing all monomials by multivariate orthogonal polynomials.
\end{abstract}
\maketitle

\section{Introduction}
Quantum groups have played a key role in many areas of mathematics and mathematical physics since their introduction by Drinfeld \cite{inp-Drinfeld1} and Jimbo \cite{a-Jimbo1}
in the 1980s. In the late 1990s Andruskiewitsch and Schneider initiated a powerful program for classifying pointed Hopf algebras \cite{MSRI-AS02}
which lead to far reaching generalizations of quantum groups, namely Drinfeld doubles of pre-Nichols algebras.

Quantum symmetric pairs in the above frameworks have become the subject of intense research. The general construction of quantum symmetric pairs in the setting of quantized enveloping algebras 
of finite dimensional semisimple Lie algebras was given by Letzter \cite{a-Letzter99}. The Kac--Moody setting was treated in \cite{a-Kolb14}. 
Quantum symmetric pairs in the setting of Drinfeld doubles of pre-Nichols algebras were defined in \cite{a-KY19p}. In those settings the quantum symmetric pairs 
having Iwasawa decompositions were characterized in \cite{a-Letzter97,a-Letzter99,a-Kolb14,a-KY19p}.

Quantum symmetric pair coideal subalgebras $B_\bc$ depend on a set of parameters $\bc=(c_i)_{i\in I}$, and are defined in terms of generators $B_i$ for $i\in I$ in the ambient Hopf algebra. The generators $B_i$ satisfy deformed quantum Serre relations, see \cite[Section 7]{a-Letzter03}, \cite[Section 7]{a-Kolb14}. One of the outstanding problems in the area of quantum symmetric pairs is to determine explicit, conceptual formulas for the deformed quantum Serre relations. The goal of this paper is the following:
\medskip

\noindent
{\bf{Metatheorem.}} {\em{The deformed quantum Serre relations for a quantum symmetric pair coideal subalgebra are obtained from the usual quantum Serre relations by replacing all monomials by multivariate orthogonal polynomials}}. 
\medskip

While in the present paper we only establish the Metatheorem in the so called quasi-split Kac--Moody setting, we expect that this phenomenon holds in full generality. Our proof is based on a result from \cite{a-KY19p} that {\em{(quantum) symmetric pair coideal subalgebras are isomorphic to star products on partial bosonizations of pre-Nichols algebras}}.

In \cite{a-Letzter03} Letzter developed a method to obtain the deformed quantum Serre relations from coproducts. She applied her method to obtain explicit relations for all quantum symmetric pairs of finite type. Letzter's method was extended to the Kac-Moody setting in \cite{a-Kolb14} and applied in the case of 
Cartan matrices $(a_{ij})_{i,j\in I}$ with $|a_{ij}|\le 3$. Recall that quantum symmetric pairs depend on an involutive diagram automorphism $\tau:I\rightarrow I$. In the case $\tau(i)=j\neq i$ the corresponding deformed quantum Serre relations were explicitly determined by Letzter's method in \cite[Theorem 3.6]{a-BalaKolb15}. In the case $\tau(i)=i\neq j$ Letzter's method gets substantially harder. Nonetheless, recently, de Clercq used Letzter's method to produce involved combinatorial formulas for deformed quantum Serre relations in the Kac--Moody case for $\tau(i)=i\neq j$ \cite{a-dC19p}. However, the connection to orthogonal polynomials is not immediate.

In the quasi-split Kac--Moody setting the deformed quantum Serre relations in the case $\tau(i)=i\neq j$ were first derived in \cite{a-CLW18p} in terms of so called $\imath$divided powers. The $\imath$divided powers are univariate polynomials and play an important role in the theory of canonical bases for quantum symmetric pairs \cite{a-BW18p}. However, an interpretation of $\imath$divided powers in terms of orthogonal polynomials is not known. In the quasi-classical limit, formulas for the corresponding deformed Serre relations were recently obtained by Stokman in \cite{a-Stok19}.

In the present paper we give an explicit expression of deformed quantum Serre relations for quasi-split quantum symmetric pairs in terms of bivariate 
continuous $q$-Hermite polynomials. Our proofs are shorter than those in previous approaches and are based on a direct relation
between the star products of  \cite{a-KY19p} and multivariate orthogonal polynomials.

The classical Hermite polynomials are the polynomials given by the recurrence relation 
\[
H_{n+1}(x) =  2x H_n(x) - 2n H_{n-1}(x),
\]
where $H_0(x)=1$. They have two types of $q$-analogs, the continuous and discrete $q$-Hermite polynomials \cite[\S 14.26-29]{koekoek2010}.
The continuous $q$-Hermite polynomials \cite[\S 14.6]{koekoek2010} satisfy the recurrence relation 
\[
H_{n+1}(x;q) =  2x H_n(x;q) - (1-q^n) H_{n-1}(x;q),
\]
where $H_0(x;q)=1$.
They appear in a number of diverse situations. For instance, recently Borodin and Corwin used them in the study of the dynamic  asymmetric simple exclusion process
\cite{a-BorCor22}. Motivated by It{\^o}'s complex bivariate orthogonal Hermite polynomials \cite{ito}, Ismail and Zhang \cite{ismail} 
defined and studied two versions of bivariate $q$-Hermite polynomials $H_{m,n}(x, y|q)$ (without additional parameters). 
They satisfy $H_{m,0}(x,y) =x^m$, $H_{0,n}(x,y) =y^n$. 

In this paper we define and study a completely different (two-parameter) family of bivariate continuous $q$-Hermite polynomials $H_{m,n}(z_1, z_2;q,r)$. They satisfy 
$$H_{m,0}(x,y;q,r) = H_m(x;q)\ \ \text{and}\ \ H_{0,n}(x,y;q,r) = H_n(y;q)$$
and are recursively defined by 
\begin{align*}
H_{m+1,n}(x,y;q,r) & = 2xH_{m,n}(x,y;q,r) - (1-q^m)H_{m-1,n}(x,y;q,r)\\
               & - q^m(1-q^n)rH_{m,n-1}(x,y;q,r).
\end{align*}
We establish algebraic and analytic properties of these polynomials. On the algebraic side, we prove that they are explicitly given by 
\begin{align}\label{eq:HmnHmHn}
H_{m,n}(x,y;q,r) = \sum_{k=0}^{\min(m,n)}\frac{(-1)^kq^{\binom{k}{2}}(q;q)_m(q;q)_nr^k}{(q;q)_{m-k}(q;q)_{n-k}(q;q)_k}H_{m-k}(x;q)H_{n-k}(y;q),
\end{align}
and in particular, they are symmetric with respect to $x$ and $y$: $H_{m,n}(x,y;q,r) = H_{n,m}(y,x;q,r)$. 
We show that their generating function is given by
$$\sum_{m,n=0}^\infty \frac{H_{m,n}(x,y;q,r)}{(q;q)_m(q;q)_n}s^mt^n  = \frac{(rst;q)_\infty}{|(se^{i\theta},te^{i\phi};q)_\infty|^2} \cdot$$
We derive an operator formulation and a backward difference equation for these polynomials (see Theorem \ref{thm:summary of bivariate q-Hermite polynomials}). 
On the analytic side we prove that they are orthogonal with respect to the measure 
\[
\frac{|(e^{2i(\alpha+\beta)}/r;q)_\infty|^2}{\sqrt{(1-x^2)(1-y^2)}}dxdy \; \; \mbox{on} \; \; [0,1] \times [0,1], \; \; 
\mbox{where} \; \; x = \cos(2\alpha), y=\cos(2\beta).
\]
We believe that these polynomials will find application outside the realm of Hopf algebras and quantum symmetric pairs.

With the above notation we can now express the deformed quantum Serre relations for quasi-split quantum symmetric pairs in the case $\tau(i)=i\neq j$. 
For $w(x,y)=\sum_{r,s}b_{rs} x^r y^s\in \field[x,y]$, set
\[
z \curvearrowright w(x, y) = \sum_{r,s}b_{rs} x^r z y^s.
\]
The following theorem is derived from the algebraic 
properties of bivariate continuous $q$-Hermite polynomials. The theorem holds for general deformation parameters $q$ including roots of unity.

\medskip

\noindent{\bf Theorem.} (Corollaries \ref{cor:dqS-bivariate}, \ref{cor:dqS-univariate}) {\em Let $i,j\in I$ and $\tau(i)=i\neq j$. Then the generators $B_i, B_j$ of the quantum symmetric pair coideal subalgebra $B_\bc$ satisfy the relation
\[
\sum_{\ell=0}^{1-a_{ij}} (-1)^{\ell} \begin{bmatrix}1-a_{ij} \\ \ell \end{bmatrix}_{q_i}
B_j \curvearrowright w_{1-a_{ij} - \ell, \ell}(B_i, B_i) =0, 
\]
where $w_{m,n}(x,y)= (2b_i)^{-m-n} H_{m,n}(b_ix,b_iy;q_i^2,q_i^{a_{ij}})$ and $b_i=\frac{1}{2}(q_i-q_i^{-1}) c_i^{-1/2}q_i^{-1/2}$. 
This relation can also be written as
\begin{align*}
   \sum_{\ell=0}^{1-a_{ij}} (-1)^{\ell} \begin{bmatrix}1-a_{ij} \\ \ell \end{bmatrix}_{q_i} w_{1-a_{ij}-\ell}(B_i) B_j v_{\ell}(B_i)= 0,
\end{align*}  
where 
\begin{align*}
  w_m(x)=  \frac{1}{(2b_i)^m}H_m(b_ix;q_i^2),\qquad
  v_m(x)=  \frac{1}{(2b_i)^m}H_m(b_ix;q_i^{-2}).
\end{align*}
}

The paper is organized as follows. Section 2 contains background material on multivariate orthogonal polynomials and the statements 
of our results on bivariate continuous $q$-Hermite polynomials. Section 3 contains the proof of these results. In 
Section 4 we recall the isomorphism theorem from \cite{a-KY19p} identifying quantum symmetric pair coideal subalgebras with star products on partial 
bosonizations of pre-Nichols algebras. Then we use the algebraic facts on the bivariate continuous $q$-Hermite polynomials to derive
the defining relations of quantum symmetric pair coideal subalgebras of quantum groups in the quasi-split Kac--Moody case.
\\ \hfill \\
\noindent
{\bf Acknowledgements.} We are grateful to the referee for the detailed comments
which helped us to improve the exposition. The research of W.R.C. was supported by a 2018 AMS-Simons Travel Grant.
The research of M.Y. was supported by NSF grant DMS-1901830 and Bulgarian Science Fund grant DN02/05.
\section{Orthogonal polynomials}\label{sec:background orthogonal polynomials}
In this section, we provide a brief review of the theory of orthogonal polynomials in a single variable and introduce the Hermite and continuous $q$-Hermite polynomials as examples.
We then recall the definition of multivariate orthogonal polynomials and introduce a bivariate analog of the continuous $q$-Hermite polynomials.
\subsection{Classical Orthogonal Polynomials}\label{sec:background classical}
A sequence of orthogonal polynomials on the real line is a sequence $p_0(x),p_1(x),\dots$ of complex-valued polynomials with $\deg p_n(x) = n$ for all $n\geq 0$, which satisfy the orthogonality condition
$$\int_{\bbr}p_m(x)\ol{p_n(x)} d\mu(x) = \delta_{m,n}c_n$$
for some positive Borel measure $\mu$ on $\bbr$ and sequence of positive constants $\{c_n\}_{n=0}^\infty$.

An elementary argument shows that any sequence of orthogonal polynomials automatically satisfies a three-term recursion relation of the form
\begin{equation}\label{three-term recursion}
xp_n(x) = \alpha_np_{n+1}(x) + \beta_np_n(x) + \gamma_np_{n-1}(x),
\end{equation}
for some sequence of constants $\{\alpha_n\}_{n=0}^\infty,\{\beta_n\}_{n=0}^\infty$ and $\{\gamma_n\}_{n=1}^\infty$, with $p_{-1}(x) := 0$.
The values are related to the moments of $\mu(x)$ and the leading coefficients of the $p_n(x)$'s.
Conversely, for any sequences $\{\alpha_n\}_{n=0}^\infty,\{\beta_n\}_{n=0}^\infty$ and $\{\gamma_n\}_{n=1}^\infty$ with $\beta_n$ real and $\alpha_n\gamma_n$ positive, the sequence of polynomials defined recursively by \eqref{three-term recursion} will be orthogonal polynomials for some Borel measure $\mu(x)$.
This result is known as Favard's theorem and is a consequence of the spectral theorem applied to the semi-infinite Jacobi matrix defined by the three-term recursion relation \cite{favard}.

The most fundamental examples of orthogonal polynomials are the classical orthogonal polynomials of Hermite, Laguerre, and Jacobi.
These polynomials satisfy the additional property that they are eigenfunctions of a second-order differential equation in the variable $x$, i.e.
$$a_2(x)p_n''(x) + a_1(x)p_n'(x) + a_0(x)p_n(x) = \lambda_np_n(x)$$
for some functions $a_0(x),a_1(x)$ and $a_2(x)$ and sequence of complex numbers $\{\lambda_n\}_{n=0}^\infty$.
Each sequence of classical orthogonal polynomials satisfies a Rodrigues-type recurrence relation and has a nice generating function formula.
For example consider the classical Hermite polynomials $H_n(x)$ defined by
\begin{equation}\label{def:hermite}
H_n(x) = n!\sum_{m=0}^{\lfloor n/2\rfloor} \frac{(-1)^m}{m!(n-2m)!}(2x)^{n-2m}.
\end{equation}
\begin{eg}
The Hermite polynomials $H_n(x)$ have the following properties \cite{koekoek2010}:
\begin{itemize}
\item  orthogonality relation:
$$\int_\bbr H_m(x)H_n(x)e^{-x^2}dx = \sqrt{\pi}2^nn!\delta_{m,n}$$
\item  three-term recursion relation:
$$xH_n(x) = \frac{1}{2}H_{n+1}(x) + nH_{n-1}(x)$$
\item  second-order differential equation:
$$H_n''(x) -2xH_n'(x) = -2nH_n(x).$$
\item  generating function:
$$e^{2xt-t^2} =\sum_{n=0}^\infty \frac{H_n(x)}{n!}t^n$$
\item  Rodrigues-type recurrence relation:
$$H_n(x) = (-1)^ne^{x^2}\left(\frac{d}{dx}\right)^n\cdot e^{-x^2}$$
\end{itemize}
\end{eg}

The classical orthogonal polynomials naturally generalize when we replace the differential operator with a second-order difference or $q$-difference operator, in which we obtain the various families obtained from the Askey and $q$-Askey scheme, such as the Wilson, Racah, Hahn, Meixner, Meixner-Pollaczek, Krawtchouk, and Charlier polynomials and their $q$-analogues.
As before, each such sequence of orthogonal polynomials satisfies a three-term recursion relation, a differential, difference, or $q$-difference equation, a Rodrigues-type recurrence relation, and has a nice generating function formula.
In this paper, we will be particularly interested in the continuous $q$-Hermite polynomials $H_n(x;q)$ defined for $x=\cos(\theta)$ by
\begin{equation}\label{def:q-hermite}
H_n(x;q) := \sum_{k=0}^n\frac{(q;q)_n}{(q;q)_k(q;q)_{n-k}}e^{i(n-2k)\theta} = e^{in\theta}\pFq{2}{0}{q^{-n},0}{-}{q,q^ne^{-2i\theta}}.
\end{equation}
Here $(a;q)_n$ is the $q$-Pochhammer symbol and $\pFq{2}{0}{a,b}{-}{q,z}$ is the $q$-hypergeometric function defined respectively by
$$(a;q)_n = \prod_{k=0}^{n-1}(1-aq^k)\ \ \ \ \text{and}\ \ \ \ \pFq{2}{0}{a,b}{-}{q,z} = \sum_{n=0}^\infty \frac{(a;q)_n(b;q)_n}{(q;q)_n(-1)^nq^{\binom{n}{2}}}z^n.$$
We also have the infinite $q$-Pochhammer symbol $(a;q)_\infty = \lim_{n\rightarrow\infty} (a;q)_n$ which has useful series expansion we will rely on later in this paper
\begin{equation}\label{q-pochhammer series}
(a;q)_\infty = \prod_{k=0}^\infty(1-aq^k) = \sum_{n=0}^\infty \frac{(-1)^n q^{\binom{n}{2}}}{(q;q)_n}a^n.
\end{equation}

Note that in terms of the Chebyshev polynomials of the first kind $T_n(x)$, this may be rewritten as
$$H_n(x;q) := \sum_{k=0}^n\frac{(q;q)_n}{(q;q)_k(q;q)_{n-k}}T_{|n-2k|}(x),$$
so that in particular $H_n(x;q)$ is a polynomial in $x$ of degree $n$ for all $n$.
\begin{eg}
The continuous $q$-Hermite polynomials $H_n(x;q)$ satisfy the following properties \cite{koekoek2010}:
\begin{itemize}
\item  orthogonality relation:
\begin{equation}\label{q-hermite orthogonality}
\int_{-1}^1 H_m(x;q)H_n(x;q)\frac{|(e^{2i\theta};q)_\infty|^2}{\sqrt{1-x^2}}dx = \frac{2\pi\delta_{m,n}}{(q^{n+1};q)_\infty}.
\end{equation}
\item  three-term recursion relation:
\begin{equation}\label{q-hermite 3-term recursion}
2xH_n(x;q) = H_{n+1}(x;q) + (1-q^n)H_{n-1}(x;q).
\end{equation}
\item (forward) $q$-difference equation:
\begin{equation}\label{q-hermite difference}
D_qH_n(x;q) = \frac{2q^{-(n-1)/2}(1-q^n)}{1-q}H_{n-1}(x;q).
\end{equation}
\item generating function:
\begin{equation}\label{q-hermite generating}
\sum_{n=0}^\infty\frac{H_n(x;q)}{(q;q)_n}s^n = \frac{1}{|(se^{i\theta};q)_\infty|^2}.
\end{equation}
\item  Rodrigues-type recurrence relation:
\begin{equation}\label{q-hermite Rodrigues}
H_n(x;q) = \left(\frac{q-1}{2}\right)^nq^{\frac{1}{4}n(n-1)}\frac{\sqrt{1-x^2}}{|(e^{2i\theta};q)_\infty|^2}(D_q)^n\cdot\frac{|(e^{2i\theta};q)_\infty|^2}{\sqrt{1-x^2}}
\end{equation}
\end{itemize}
\end{eg}

\begin{rema}
In the $q$-difference equation above, $D_q$ is the $q$-difference operator found in \cite[Equation 1.16.4]{koekoek2010}, given by
$$D_q f(x) = \frac{\delta_q f(x)}{\delta_q x},\ \ \ \ x=\cos(\theta)$$
where here
$$\delta_q f(e^{i\theta}) = f(q^{1/2}e^{i\theta})-f(q^{-1/2}e^{i\theta}),$$
so that in particular $\delta_q x = -\frac{1}{2}q^{-1/2}(1-q)(e^{i\theta}-e^{-i\theta})$ for $x = \cos\theta$.
\end{rema}

\subsection{Multivariate orthogonal polynomials}\label{sec:background multivariate}
The theory of multivariate orthogonal polynomials on $\bbr^d$ is considerably more complicated than the single variable situation and far less complete.
Even so, the basics of the theory remain the same as long as the definitions are taken appropriately.
Some useful introductory references are \cite{dunkl2014,xu2005}.

For simplicity, we will adopt the vector notation $\vec x = (x_1,\dots, x_d)$ and $\vec n = (n_1,\dots, n_d)$ and will write $|\vec n|$ to mean $n_1+\dots+n_d$.
We will also use the monomial notation $x^{\vec n}$ for the product $x_1^{n_1}x_2^{n_2}\dots x_d^{n_d}$.
\begin{defi}
A sequence of orthogonal polynomials in $d$ variables is a sequence $p_{\vec n}(\vec x)$ of polynomials in variables $x_1,\dots, x_d$ such that
\begin{enumerate}[(a)]
\item for all $m$ the polynomials $\{p_{\vec n}(\vec x): |\vec n| \leq m\}$ define a basis for the space of polynomials of total degree at most $m$
\item there exists a positive Borel measure $\mu$ on $\bbr^d$ with finite moments $\int_{\bbr}|x^{\vec n}|d\mu(\vec x) <\infty$ satisfying
$$\int_{\bbr^d} p_{\vec m}(\vec x)\ol{p_{\vec n}(\vec x)} d\mu(\vec x) = 0\ \ \text{for $|\vec m|\neq |\vec n|$}.$$
\end{enumerate}
\end{defi}
In other words, polynomials of different total degrees are orthogonal, but different polynomials with the same total degree may not be.
In particular, one may have to perform a change of basis
$$\wt p_{\vec n}(x,y) = \sum_{|\vec m|=|\vec n|} a_{\vec m} p_{\vec m}(\vec x),$$
to get a sequence of orthogonal polynomials satisfying the more intuitive orthogonality condition
\begin{equation}\label{naive orthogonality}
\int_{\bbr}\wt p_{\vec m}(\vec x)\ol{\wt p_{\vec{n}}}(\vec x)d\mu(\vec x) = 0\ \ \text{when $\vec m\neq \vec n$}.
\end{equation}
to apply for all $\vec m$ and $\vec n$.

A sequence of orthogonal polynomials again gives rise to a three-term recursion relation, except that the summands are in terms of the total degree and can involve multiple polynomials with the same total degree.
Specifically, there will exist constants $\alpha_{\vec n,\vec m,j},\beta_{\vec n,\vec m,j},\gamma_{\vec n,\vec m,j}$ such that for all $j=1,\dots,d$
$$x_jp_{\vec n}(\vec x) = \sum_{|\vec m| = |\vec n| + 1}\alpha_{\vec n,\vec m,j}p_{\vec m}(\vec x) + \sum_{|\vec m| = |\vec n|} \beta_{\vec n,\vec m,j}p_{\vec m}(\vec x) + \sum_{|\vec m|=|\vec n|-1}\gamma_{\vec n,\vec m,j}p_{\vec m}(\vec x).$$
An analog of Favard's theorem has also been proved \cite{xu1993}, i.e.~for sufficiently nice sequences of constants, the sequence of polynomials defined by the three-term recursion relation will be orthogonal with respect to some measure $\mu$ on $\bbr^d$.
As mentioned above, we can then change our basis so that the orthogonal polynomials satisfy the simple orthogonality condition \eqref{naive orthogonality}, but this in turn will completely change the original recurrence relations and the new orthogonal polynomials may lose other desirable properties such as having monomial leading coefficients.

In the next section, we will construct two dimensional analogs of the continuous $q$-Hermite polynomials defined above, which we will hereafter refer to as the bivariate continuous $q$-Hermite polynomials
\begin{equation}\label{bivariate q-Hermite equation}
H_{m,n}(x,y;q,r) = \sum_{k=0}^{\min(m,n)}\frac{(-1)^kq^{\binom{k}{2}}(q;q)_m(q;q)_nr^k}{(q;q)_{m-k}(q;q)_{n-k}(q;q)_k}H_{m-k}(x;q)H_{n-k}(y;q).
\end{equation}
Note that
$$H_{m,0}(x,y;q,r) = H_m(x;q)\ \ \text{and}\ \ H_{0,n}(x,y;q,r) = H_n(y;q).$$
Moreover
$$H_{m,n}(x,y;q,0) = H_m(x;q)H_n(y;q),$$
so $H_{m,n}(x,y;q,r)$ may be thought of as a deformation of the family of orthogonal polynomials $H_m(x;q)H_n(y;q)$ with deformation parameter $r$.
\begin{rema}
Our bivariate continuous $q$-Hermite polynomials are very different from those constructed by Ismail and Zhang \cite{ismail}, which were motivated by the complex bivariate orthogonal Hermite polynomials introduced by It{\^o} \cite{ito}.
\end{rema}
\begin{rema}
We do not define the bivariate continuous $q$-Hermite polynomials with \eqref{bivariate q-Hermite equation}.
Instead, we define them in the next section in terms of a symmetry condition and a three-term recursion relation reminiscent of the recursion relation for the one variable case.
We then prove that the resulting sequence satisfies \eqref{bivariate q-Hermite equation}.
\end{rema}
In the next section we will prove several important properties of these polynomials, including orthogonality, recurrence relations, $q$-difference equations, and a generating function formulation.
We summarize these properties here for the convenience of the reader.
\begin{thm}\label{thm:summary of bivariate q-Hermite polynomials}
The bivariate continuous $q$-Hermite polynomials $H_{m,n}(x,y;q,r)$ satisfy the following properties.
\begin{itemize}
\item  Orthogonality relation:
$$\int_{-1}^1\int_{-1}^1H_{m,n}(x,y;q,r)H_{\wt n,\wt m}(x,y;q,r)\frac{|(e^{2i(\alpha+\beta)}/r;q)_\infty|^2}{\sqrt{(1-x^2)(1-y^2)}}dxdy = c_{m,n}\delta_{m,\wt m}\delta_{n,\wt n},$$
where here $x = \cos(2\alpha)$, $y=\cos(2\beta)$ and
$$c_{m,n} = \frac{2\pi^2}{(q;q)_\infty}\sum_{\substack{i+k+\ell=m\\j+k+\ell=n}}\frac{(-1)^kq^{\binom{k}{2}}(q;q)_m^2(q;q)_n^2}{(q;q)_i(q;q)_j(q;q)_k(q;q)_\ell^2}r^{m+n}.$$
\item Three-term recursion relations:
\begin{align*}
2xH_{m,n}(x,y;q,r) = H_{m+1,n}(x,y;q,r) & + (1-q^m)H_{m-1,n}(x,y;q,r)\\\nonumber
               & + q^m(1-q^n)rH_{m,n-1}(x,y;q,r),\\
2yH_{m,n}(x,y;q,r) = H_{m,n+1}(x,y;q,r) & + (1-q^n)H_{m,n-1}(x,y;q,r)\\\nonumber
               & + q^n(1-q^m)rH_{m-1,n}(x,y;q,r).
\end{align*}
\item Generating function:
$$\sum_{m,n=0}^\infty \frac{H_{m,n}(x,y;q,r)}{(q;q)_m(q;q)_n}s^mt^n  = \frac{(rst;q)_\infty}{|(se^{i\theta},te^{i\phi};q)_\infty|^2} \cdot$$
\item  Operator formulation:
$$H_{m,n}(x,y;q,r) =  \frac{1}{\left(-q^{\frac{m+n}{2}-1}\left(\frac{1-q}{2}\right)^2rD_{q,x}D_{q,y};q\right)_\infty}\cdot H_m(x;q)H_n(y;q).$$
\item  Additional relations:
\begin{align*}
D_{q,x}\cdot H_{m,n}(x,y;q,r) &= \frac{2q^{-\frac{m-1}{2}}(1-q^m)}{1-q}H_{m-1,n}(x,y;q,\sqrt{q}r),\\
D_{q,y}\cdot H_{m,n}(x,y;q,r) &= \frac{2q^{-\frac{n-1}{2}}(1-q^n)}{1-q}H_{m,n-1}(x,y;q,\sqrt{q}r).
\end{align*}
\end{itemize}
\end{thm}

\section{Generating function,  recursion relations and orthogonality}\label{sec:recursion}
In this section we define the bivariate continuous $q$-Hermite polynomials and prove the properties stated in Theorem \ref{thm:summary of bivariate q-Hermite polynomials}.
Excepting the initial definition of the bivariate continuous $q$-Hermite polynomials below, we will write $H_{m,n}(x,y)$ in place of $H_{m,n}(x,y;q,r)$ throughout this section for sake of brevity.
\subsection{The bivariate continuous $q$-Hermite polynomials}\label{sec:bivariate}
As a two-dimensional analog of the continuous $q$-Hermite polynomials, we consider the following sequence of bivariate polynomials.
\begin{defi}\label{defi:bivariate hermite}
The bivariate continuous $q$-Hermite polynomials are the unique sequence of orthogonal polynomials $H_{m,n}(x,y;q,r)$ defined for all integers $m,n\geq 0$ satisfying the symmetry condition
\begin{equation}\label{poly symmetry}
H_{m,n}(x,y;q,r) = H_{n,m}(y,x;q,r)
\end{equation}
as well as the three-term recursion relation
\begin{align}\label{eq:Hmn-recursion}
2xH_{m,n}(x,y;q,r)
  & = H_{m+1,n}(x,y;q,r)\\\nonumber
  & + (1-q^m)H_{m-1,n}(x,y;q,r)\\\nonumber
  & + q^m(1-q^n)rH_{m,n-1}(x,y;q,r),
\end{align}
with $H_{0,0}(x,y;q,r) = 1$ and $H_{-1,0}(x,y;q,r) = H_{0,-1}(x,y;q,r) = 0$.
\end{defi}
Note in particular $H_{m,n}(x,y;q,r)$ is a polynomial of bidegree $(m,n)$, and that $H_{m,0}(x,y;q,r) = H_m(x;q)$.

Mimicking the generating function in the single-variable case \eqref{q-hermite generating}, we consider the function
$$\psi\binom{x,y}{s,t} :=  \sum_{m,n=0}^\infty \psi_{m,n}\binom{x,y}{s,t},\ \ \text{where}\ \ \psi_{m,n}\binom{x,y}{s,t} := \frac{H_{m,n}(x,y)s^mt^n}{(q;q)_m(q;q)_n}.$$
Note that the recursion relation above tells us
\begin{align*}
2x\psi_{m,n}\binom{x,y}{s,t}
  & = \frac{1}{s}\left(\psi_{m+1,n}\binom{x,y}{s,t} - \psi_{m+1,n}\binom{x,y}{qs,t}\right)\\
  & + s\psi_{m-1,n}\binom{x,y}{s,t} + rt\psi_{m,n-1}\binom{x,y}{qs,t}.
\end{align*}
Summing this, we find
$$2x\psi\binom{x,y}{s,t} = \frac{1}{s}\left(\psi\binom{x,y}{s,t} - \psi\binom{x,y}{qs,t}\right) + s\psi\binom{x,y}{s,t} + rt\psi\binom{x,y}{qs,t},$$
which simplifies to the homogeneous $q$-difference equation
$$\psi\binom{x,y}{qs,t} = \left(\frac{2xs-s^2-1}{rst-1}\right)\psi\binom{x,y}{s,t}.$$
Factoring
$$\left(\frac{2xs-s^2-1}{rst-1}\right) = \frac{(1-(x+\sqrt{x^2-1})s)(1-(x-\sqrt{x^2-1})s)}{1-rst}$$
and using the fact that $(qx;q)_\infty(1-x) = (x;q)_\infty$, we see that the general solution of this $q$-difference equation is
\begin{align*}
\psi\binom{x,y}{s,t}
  & = \frac{(rst;q)_\infty}{((x + \sqrt{x^2-1})s;q)_\infty((x-\sqrt{x^2-1})s;q)_\infty}\psi\binom{x,y}{0,t}\\
  & = \frac{(rst;q)_\infty}{|(e^{i\theta}s;q)_\infty|^2}\psi\binom{x,y}{0,t},\ \ \text{for $x = \cos(\theta)$}.
\end{align*}
Finally by symmetry and the choice that $H_{0,0}(x,y) = 1$, or alternatively by using \eqref{q-hermite generating}, we obtain a generating function formula for the polynomials $H_{m,n}(x,y)$ with $x = \cos(\theta)$ and $y = \cos(\phi)$
\begin{equation}\label{bivariate q-hermite generating}
\psi\binom{x,y}{s,t} := \sum_{m,n=0}^\infty \frac{H_{m,n}(x,y)}{(q;q)_m(q;q)_n}s^mt^n  = \frac{(rst;q)_\infty}{|(se^{i\theta},te^{i\phi};q)_\infty|^2}.
\end{equation}

The generating function equation also allows us to express our bivariate continuous $q$-Hermite polynomials in terms of the continuous $q$-Hermite polynomials in a single variable.
By applying \eqref{q-hermite generating} along with the series expansion for the $q$-Pochhammer symbol \eqref{q-pochhammer series} for $(rst;q)_\infty$ we see
$$\frac{(rst;q)_\infty}{|(se^{i\theta},te^{i\phi};q)_\infty|^2} = \sum_{k,m,n=0}^\infty \frac{(-1)^kq^{\binom{k}{2}}r^ks^{m+k}t^{m+k}}{(q;q)_k(q;q)_m(q;q)_n}H_m(x)H_n(y).$$
Comparing similar powers of $s$ and $t$, we find
\begin{equation}\label{bivariate in terms of univariate}
H_{m,n}(x,y) = \sum_{k=0}^{\min(m,n)}\frac{(-1)^kq^{\binom{k}{2}}(q;q)_m(q;q)_nr^k}{(q;q)_{m-k}(q;q)_{n-k}(q;q)_k}H_{m-k}(x;q)H_{n-k}(y;q).
\end{equation}

\subsection{Orthogonality}
By Xu's extension of Favard's theorem \cite{xu1993}, we expect a sequence of multivariate polynomials with $n$ variables which satisfies a sufficiently nice three-term recursion relation to be orthogonal with respect to some inner product defined by a measure on $\bbr^n$.
This is indeed the case for the bivariate continuous $q$-Hermite polynomials we defined, as we prove in the following theorem.
\begin{thm}\label{thm:orthogonality}
The bivariate continuous $q$-Hermite polynomials satisfy the orthogonality relation
$$\int_{-1}^1\int_{-1}^1H_{m,n}(u,v;q,r)H_{\wt n,\wt m}(u,v;q,r)\frac{|(e^{2i(\alpha+\beta)}/r;q)_\infty|^2}{\sqrt{(1-u^2)(1-v^2)}}dudv = c_{m,n}\delta_{m,\wt m}\delta_{n,\wt n},$$
where here $u = \cos(2\alpha)$, $v = \cos(2\beta)$ and 
\begin{equation}\label{norm squared equation}
c_{m,n} = \frac{2\pi^2}{(q;q)_\infty}\sum_{\substack{i+k+\ell=m\\j+k+\ell=n}}\frac{(-1)^kq^{\binom{k}{2}}(q;q)_m^2(q;q)_n^2}{(q;q)_i(q;q)_j(q;q)_k(q;q)_\ell^2}r^{m+n}.
\end{equation}
\end{thm}
\begin{rema}
Note that $(\alpha,\beta)\rightarrow (\cos(2\alpha),\cos(2\beta))$ defines a fourfold cover from the diamond region $D$ with vertices $(0,0),(\pi/2,\pi/2),(\pi/2,-\pi/2)$ and $(\pi,0)$ to the triangular $T$ with vertices $(-1,-1),(-1,1),$ and $(1,1)$.
Furthermore, the map $\theta = \alpha+\beta$ and $\phi=\alpha-\beta$ maps $D$ to the square region $[0,\pi]^2$.
Thus if $I_{m,n,\wt m,\wt n}$ is the integral in Theorem \ref{thm:orthogonality}, we have by symmetry
\begin{align*}
I_{m,n,\wt m,\wt n}  & = 2\iint\limits_{T}H_{m,n}(u,v;q,r)H_{\wt n,\wt m}(u,v;q,r)\frac{|(e^{2i(\alpha+\beta)}/r;q)_\infty|^2}{\sqrt{(1-u^2)(1-v^2)}}dudv\\
   & = 2\iint\limits_{D}H_{m,n}(u,v;q,r)H_{\wt n,\wt m}(u,v;q,r)|(e^{2i(\alpha+\beta)}/r;q)_\infty|^2d\alpha d\beta\\
   & = \int_0^\pi\int_0^\pi H_{m,n}(u,v;q,r)H_{\wt n,\wt m}(u,v;q,r)|(e^{2i\theta}/r;q)_\infty|^2d\theta d\phi.
\end{align*}
where here $u=\cos(2\alpha)$, $v=\cos(2\beta)$, $\theta = \alpha+\beta$ and $\phi = \alpha-\beta$.
Thus the orthogonality expression above is equivalent to
\begin{equation}
\int_0^\pi\int_0^\pi H_{m,n}(u,v;q,r)H_{\wt n,\wt m}(u,v;q,r)|(e^{2i\theta}/r;q)_\infty|^2d\theta d\phi = c_{m,n}\delta_{m,\wt m}\delta_{n,\wt n},
\end{equation}
for $u = \cos(\theta+\phi)$ and $v = \cos(\theta-\phi)$.
\end{rema}
\begin{rema}
If we define a new sequence of polynomials
$$\wt H_{m,n}(x,y) = \left\lbrace\begin{array}{cc}
H_{m,n}(x,y)-H_{n,m}(x,y)& \text{if $m<n$},\\
H_{m,n}(x,y)+H_{n,m}(x,y)& \text{if $m\geq n$}
\end{array}\right.$$
then the new sequence satisfies the more intuitive orthogonality statement that
$$
\int_0^\pi\int_0^\pi \wt H_{m,n}(u,v;q,r)\wt H_{\wt m,\wt n}(u,v;q,r)|(e^{2i\theta}/r;q)_\infty|^2d\theta d\phi = \wt c_{m,n}\delta_{m,\wt m}\delta_{n,\wt n}
$$
for some constants $\wt c_{m,n} > 0$, but will no longer have monomial leading coefficients.
\end{rema}

\begin{proof}
To prove Theorem \ref{thm:orthogonality}, we will use the generating function formula for the bivariate continuous $q$-Hermite polynomials to deduce an orthogonality condition.
We will also make use of the Askey--Wilson integral \cite[Theorem 10.8.1]{andrews1999}
\begin{equation}
\int_{-\pi}^\pi\frac{|(e^{2i\theta};q)_\infty|^2}{|(ae^{i\theta},be^{i\theta},ce^{i\theta},de^{i\theta};q)_\infty|^2}d\theta = \frac{4\pi(abcd;q)_\infty}{(ab,ac,ad,bc,bd,cd,q;q)_\infty}.
\end{equation}
However, we require this integral in a slightly modified form.
Note that the function
$$f(z) = \frac{|(e^{2iz}/r;q)_\infty|^2}{|(ae^{iz},be^{iz},ce^{iz},de^{iz};q)_\infty|^2}$$
is, $2\pi$-periodic and holomorphic on the domain $\text{Im}(z) > \ln\max(|a|,|b|,|c|,|d|)$.
Therefore by the Cauchy residue theorem, as long as $r>\max(|a|,|b|,|c|,|d|)^2$ we have no poles in the rectangle $[-\pi,\pi]\times[-(i/2)\ln(r),0]$ and so
\begin{align*}
\int_{-\pi}^\pi f(\theta) d\theta
  & = \int_{-\pi}^\pi f(\theta-(i/2)\ln r)d\theta \\
  & + i\int_{0}^{(1/2)\ln r} f(\pi-ix)dx - i\int_{0}^{(1/2)\ln r} f(-\pi-ix)dx\\
  & = \int_{-\pi}^\pi f(\theta-(i/2)\ln r)dx.
\end{align*}
Consequently for $r>\max(|a|,|b|,|c|,|d|)^2$ we see that
\begin{equation}
\int_{-\pi}^\pi\frac{|(e^{2i\theta}/r;q)_\infty|^2}{|(ae^{i\theta},be^{i\theta},ce^{i\theta},de^{i\theta};q)_\infty|^2}d\theta = \frac{4\pi(abcdr^2;q)_\infty}{(abr,acr,adr,bcr,bdr,cdr,q;q)_\infty}.
\end{equation}
Using this, assume $\max(|s|,|t|,|\wt s|,|\wt t|)^2 < r$ and consider the integral
$$I = I(s,t,\wt s,\wt t;q,r) = \int_{0}^\pi\int_{0}^\pi\psi\binom{u,v}{s,t}\psi\binom{v,u}{\wt s,\wt t}|(e^{2i\theta}/r;q)_\infty|^2d\theta d\phi.$$
We calculate
\begin{align*}
I
  &= \int_{0}^\pi\int_{0}^\pi\frac{(rst,r\wt s\wt t;q)_\infty |(e^{2i\theta}/r;q)_\infty|^2}{|(e^{i(\theta+\phi)}s,e^{i(\theta-\phi)}t,e^{i(\theta-\phi)}\wt s, e^{i(\theta+\phi)}\wt t;q)_\infty|^2}d\theta d\phi\\
  &= \frac{1}{2}\int_{0}^\pi\int_{-\pi}^\pi\frac{(rst,r\wt s\wt t;q)_\infty |(e^{2i\theta}/r;q)_\infty|^2}{|(e^{i(\theta+\phi)}s,e^{i(\theta-\phi)}t,e^{i(\theta-\phi)}\wt s, e^{i(\theta+\phi)}\wt t;q)_\infty|^2} d\theta d\phi\\
  &= \frac{2\pi}{(q;q)_\infty}\int_{0}^\pi\frac{(rst,r\wt s\wt t,r^2s\wt st\wt t;q)_\infty}{(rst,rs\wt s,rs\wt te^{2i\phi},r\wt ste^{-2i\phi},rt\wt t,r\wt s\wt t;q)_\infty}d\phi\\
  &= \frac{2\pi(r^2s\wt st\wt t;q)_\infty}{(rs\wt s,rt\wt t,q;q)_\infty}\int_{0}^\pi\frac{1}{(rs\wt te^{2i\phi},r\wt ste^{-2i\phi};q)_\infty}d\phi\\
  &= \frac{2\pi(r^2s\wt st\wt t;q)_\infty}{(rs\wt s,rt\wt t,q;q)_\infty}\sum_{m,n=0}^\infty \int_{0}^\pi\frac{(rs\wt t)^m(r\wt st)^n}{(q;q)_m(q;q)_n}e^{2i\phi(m-n)}d\phi\\
  &= \frac{2\pi^2(r^2s\wt st\wt t;q)_\infty}{(rs\wt s,rt\wt t,q;q)_\infty}\sum_{n=0}^\infty \frac{(r^2s\wt t\wt st)^n}{(q;q)_n^2}\\
  &= \frac{2\pi^2}{(q;q)_\infty}\sum_{i,j,k,\ell} \frac{(-1)^kq^{\binom{k}{2}}(rs\wt s)^{i+k+\ell}(rt\wt t)^{j+k+\ell}}{(q;q)_i(q;q)_j(q;q)_k(q;q)_\ell^2}\\
  &= \sum_{m,n=0}^\infty \frac{c_{m,n}}{(q;q)_m^2(q;q)_n^2}(s\wt s)^m(t\wt t)^n,
\end{align*}
where the $c_{m,n}$'s are given by \eqref{norm squared equation}.
Furthermore, using the explicit series expression for $\psi\binom{x,y}{s,t}$ in terms of the continuous bivariate $q$-Hermite polynomials, we see
$$I = \sum_{m,n=0}^\infty\sum_{\wt m,\wt n=0}^\infty\frac{s^mt^n{\wt s}^{\wt m}{\wt t}^{\wt n}}{(q;q)_m(q;q)_n(q;q)_{\wt m}(q;q)_{\wt n}} I_{m,n,\wt m,\wt n}$$
for
$$I_{m,n,\wt m,\wt n} =  \int_{0}^\pi\int_{0}^\pi H_{m,n}(u,v)H_{\wt m,\wt n}(v,u)|(e^{2i\theta}/r;q)_\infty|^2d\theta d\phi.$$
Combining this with our previous expression for $I$, we find that
$$I_{m,n,\wt m,\wt n} = \delta_{m,\wt m}\delta_{n,\wt n}c_{m,n}.$$
We prove below that the constants $c_{m,n}$ are positive when $r$ and $q$ are real.
This proves the statement of Theorem \ref{thm:orthogonality}.
\end{proof}

\subsection{Complex interpretation}
The swapping of the variables $u$ and $v$ in the inner product expression of Theorem \ref{thm:orthogonality} is somewhat startling!
However, it is quite natural when viewed in terms of an inner product on the $2$-torus
$$\bbt^2 = \{(z,w)\in\bbc^2: |z| = |w| = 1\}.$$
To see what we mean specifically, consider the complex functions $\theta_{m,n}(z,w)$ defined on $\bbt^2$ by
\begin{align*}
\theta_{m,n}(z,w) = \sum_{k=0}^{\min(m,n)}&\left(\frac{(-1)^kq^{\binom{k}{2}}(q;q)_m(q;q)_n}{(q;q)_{m-k}(q;q)_{n-k}(q;q)_k}r^kz^{m-k}w^{n-k}\right.\\
                  &\left.\times \pFq{2}{0}{q^{k-m},0}{-}{q,q^{m-k}z^{-2}}\pFq{2}{0}{q^{k-n},0}{-}{q,q^{n-k}w^{-2}}\right).
\end{align*}
These are complex bivariate trigonometric polynomials on $\bbt^2$.
Note in particular
$$\theta_{m,n}(e^{i\theta},e^{i\phi}) = H_{m,n}(x,y),\ \ \text{for $x =\cos(\theta)$ and $y=\cos(\phi)$}.$$
With this in mind the inner product expression of Theorem \ref{thm:orthogonality} becomes
\begin{equation}
\frac{1}{4}\iint\limits_{\bbt^2}\theta_{m,n}(zw,z\ol w)\ol{\theta_{\wt m,\wt n}(zw,z\ol w)}|(z^2/r;q)_\infty|^2d\lvert z\rvert d\lvert w\rvert = c_{m,n}\delta_{m,\wt m}\delta_{n,\wt n}.
\end{equation}
In this way, we can see that the inner product expression from Theorem \ref{thm:orthogonality} is actually a Hermitian inner product on $L^2(\bbt^2)$.
In particular, when $r$ and $q$ are real the coefficients $c_{m,n}$ are necessarily positive as they are Hermitian inner products of polynomials with respect to an absolutely continuous positive measure whose support contains a dense open subset of $\bbt^2$.

\subsection{Additional properties}
Equation \eqref{bivariate in terms of univariate} combined with the $q$-difference equation for the Hermite polynomials \eqref{q-hermite difference} immediately tells us a simple operator identity relating the polynomials $H_{m,n}(x,y)$ to $H_m(x;q)H_n(y;q)$.
Specifically, we can write $H_{m,n}(x,y)$ as a certain differential operator of infinite order acting on the product $H_m(x;q)H_n(x;q)$, namely
\begin{equation}\label{operator relation}
H_{m,n}(x,y;q,r) =  \frac{1}{\left(-q^{\frac{m+n}{2}-1}\left(\frac{1-q}{2}\right)^2rD_{q,x}D_{q,y};q\right)_\infty}\cdot H_m(x;q)H_n(y;q).
\end{equation}
To see this, note that
$$D_{q,x}^kH_m(x;q) = \frac{2^k}{(1-q)^k}q^{-\frac{1}{2}\left(mk-\binom{k}{2}-k\right)}\frac{(q;q)_m}{(q;q)_{m-k}}H_{m-k}(x;q),$$
so that
\begin{align*}
D_{q,x}^k&D_{q,y}^kH_m(x;q)H_n(y;q)\\
  & = \frac{2^{2k}}{(1-q)^{2k}}q^{-\left(\frac{m+n}{2}k-\binom{k}{2}-k\right)}\frac{(q;q)_m(q;q)_n}{(q;q)_{m-k}(q;q)_{n-k}}H_{m-k}(x;q)H_{n-k}(x;q).
\end{align*}
Therefore
\begin{align*}
H_{m,n}(x,y)
  & = \sum_{k=0}^{\min(m,n)} \frac{\left(-q^{\frac{m+n}{2}-1}\left(\frac{1-q}{2}\right)^2rD_{q,x}D_{q,y}\right)^k}{(q;q)_k}\cdot H_m(x;q)H_n(y;q)\\
  & = \sum_{k=0}^{\infty} \frac{\left(-q^{\frac{m+n}{2}-1}\left(\frac{1-q}{2}\right)^2rD_{q,x}D_{q,y}\right)^k}{(q;q)_k}\cdot H_m(x;q)H_n(y;q)\\
  & =  \frac{1}{\left(-q^{\frac{m+n}{2}-1}\left(\frac{1-q}{2}\right)^2rD_{q,x}D_{q,y};q\right)_\infty}\cdot H_m(x;q)H_n(y;q).
\end{align*}
Note that the second equality is due to the fact that $D_{q,x}^kD_{q,y}^k H_m(x;q)H_n(y;q)$ is zero for $k>\min(m,n)$, so all the additional terms appearing in the sum are just zero.

The generating function formula \eqref{bivariate q-hermite generating} along with the operator formula \eqref{operator relation} combined with properties of the continuous $q$-Hermite polynomials in the single-variable case, immediately guarantee certain nice recurrence relations for the bivariate continuous $q$-Hermite polynomials.
We list some of these in the next proposition.
\begin{prop}
The bivariate continuous $q$-Hermite polynomials satisfy the following equations
\begin{align}
D_{q,x}\cdot H_{m,n}(x,y;q,r) &= \frac{2q^{-\frac{m-1}{1}}(1-q^m)}{1-q}H_{m-1,n}(x,y;q,\sqrt{q}r),\\
D_{q,y}\cdot H_{m,n}(x,y;q,r) &= \frac{2q^{-\frac{n-1}{1}}(1-q^n)}{1-q}H_{m,n-1}(x,y;q,\sqrt{q}r).
\end{align}
\end{prop}
\begin{proof}
The forward difference equations follow from the operator relation \eqref{operator relation}.
In detail, define
$$L_k(q,r) = \frac{1}{\left(-q^{k/2-1}\left(\frac{1-q}{2}\right)^2rD_{q,x}D_{q,y};q\right)_\infty}$$
and notice that $L_k(q,r) = L_{k-1}(q,\sqrt{q}r)$.
Therefore by \eqref{operator relation}
\begin{align*}
D_{q,x}\cdot H_{m,n}(x,y)
  & = D_{q,x}L_{m+n}(q,r)\cdot H_m(x;q)H_n(y;q)\\
  & = q^{-\frac{m-1}{2}}\left(\frac{2}{1-q}\right)(1-q^m)L_{m+n}(q,r)\cdot H_{m-1}(x;q)H_n(y;q)\\
  & = q^{-\frac{m-1}{2}}\left(\frac{2}{1-q}\right)(1-q^m)L_{m+n-1}(q,\sqrt{q}r)\cdot H_{m-1}(x;q)H_n(y;q)\\
  & = q^{-\frac{m-1}{2}}\left(\frac{2}{1-q}\right)(1-q^m)H_{m-1,n}(x,y;q,\sqrt{q}r).
\end{align*}
The proof of the other difference equation is similar.
\end{proof}
\section{Defining relations for quantum symmetric pairs}
We now explain how bivariate continuous $q$-Hermite polynomials appear in the theory of quantum symmetric pairs.
\subsection{Quasi-split quantum symmetric pairs}\label{sec:qs-qsp}
Let $\gfrak$ be a symmetrizable Kac-Moody algebra with generalized Cartan matrix $(a_{ij})_{i,j\in I}$ where $I$ is a finite set. Let $\{d_i\,|\,i\in I\}$ be a set of relatively prime positive integers such that the matrix $(d_ia_{ij})$ is symmetric. Let $\Pi=\{\alpha_i\,|\,i\in I\}$ be the set of simple roots for $\gfrak$ and let $Q=\Z\Pi$ be the root lattice. Consider the symmetric bilinear form $(\cdot,\cdot):Q\times Q \rightarrow \Z$ defined by $(\alpha_i,\alpha_j)=d_ia_{ij}$ for all $i,j\in I$. Let $\gfrak'=[\gfrak,\gfrak]$ be the derived subalgebra of $\gfrak$. We now recall the definition of the corresponding quantized enveloping algebra.

Let $\field$ be a field of characteristic zero and let $q\in \field^\times$ such that $q^{2d_i}\neq 1$ for all $i\in I$. Recall the symmetric $q$-numbers, $q$-factorials and $q$-binomial coefficients defined by
\begin{align*}
  [n]_q&=\frac{q^n-q^{-n}}{q-q^{-1}}, \qquad [n]_q^!=[n]_q [n-1]_q \cdots [2]_q [1]_q,\qquad \begin{bmatrix}n \\ m \end{bmatrix}_{q}=\frac{[n]^!_q}{[n-m]_q^!\, [m]_q^!}
\end{align*}  
for any $m,n\in\N$ with $m\le n$, see for instance in \cite[1.3.3]{b-Lusztig94}.
We abbreviate $q_i=q^{d_i}$ for any $i\in I$. For any $i,j\in I$ let $S_{ij}(x,y)$ denote the noncommutative polynomial in variables $x,y$ given by
\begin{align*}
  S_{ij}(x,y)=\sum_{n=0}^{1-a_{ij}} (-1)^n \begin{bmatrix}1-a_{ij}\\ n \end{bmatrix}_{q_i} x^{1-a_{ij}-n}y x^n.
\end{align*}  
Define $U_q(\gfrak')$ to be the $\field$-algebra with generators $E_i, F_i, K_i^{\pm 1}$ for $i\in I$ and defining relations
\begin{align}
  K_i K_j=K_j K_i,\quad K_i E_j &= q^{-(\alpha_i,\alpha_j)}E_j K_i, \quad K_iF_j = q^{-(\alpha_i,\alpha_j)}F_j K_i,\nonumber\\
  E_i F_j - F_j E_i &= \delta_{ij}\frac{K_i-K_i^{-1}}{q_i-q_i^{-1}},\label{eq:EFFE}\\
  S_{ij}(E_i,E_j)&=S_{ij}(F_i,F_j)=0 \label{eq:q-Serre}
\end{align}  
for all $i,j\in I$. The relations \eqref{eq:q-Serre} are known as the quantum Serre relations. If $q$ is not a root of unity, then $U_q(\gfrak')$ is the quantized universal enveloping algebra of $\gfrak'$ for the deformation parameter $q$ as defined in \cite{b-Lusztig94}. If $q$ is a root of unity, then $U_q(\gfrak')$ is the big quantum group of $\gfrak'$ at $q$, defined and studied by De Concini and Kac \cite{a-dCK90}. In either case $U_q(\gfrak')$ is a Hopf algebra with coproduct $\kow$ defined for all $i\in I$ by
\begin{align*}
  \kow(E_i)=E_i \ot 1 + K_i \ot E_i, \quad \kow(F_i)=F_i \ot K_i^{-1}+ 1 \ot F_i, \quad \kow(K_i)=K_i\ot K_i.
\end{align*}  
Let $\tau:I\rightarrow I$ be a bijection such that  $a_{\tau(i)\tau(j)}=a_{ij}$ for all $i,j\in I$. The diagram automorphism $\tau$ gives rise to a Lie algebra automorphism $\tau:\gfrak'\rightarrow \gfrak'$ denoted by the same symbol. Let $\omega:\gfrak'\rightarrow \gfrak'$ be the Chevalley involution as defined in \cite[(1.3.4)]{b-Kac1}. Consider the involutive Lie algebra automorphism $\theta=\tau\circ \omega$ of $\gfrak'$ and let $\kfrak'=\{x\in \gfrak'\,|\,\theta(x)=x\}$ denote the corresponding pointwise fixed Lie subalgebra. The theory of quantum symmetric pairs provides quantum group analogs of the universal enveloping algebra $U(\kfrak')$ as coideal subalgebras of $U_q(\gfrak')$. More precisely, let $H_\theta\subset U_q(\gfrak')$ denote the Hopf subalgebra generated by the elements $K_i K_{\tau(i)}^{-1}$ for all $i\in I$.
Let $\bc=(c_i)_{i\in I}\in (\field^\times)^I$ be a family of parameters such that
\begin{align}\label{eq:cond-c}
  c_i=c_{\tau(i)} \mbox{ for all $i\in I$ with $a_{i\tau(i)}=0$.}
\end{align}
We define $B_\bc$ to be the subalgebra of $U_q(\gfrak')$ generated by $H_\theta$ and the elements
\begin{align}\label{eq:Bi-def}
  B_i= F_i -c_i E_{\tau(i)} K_i^{-1} \qquad \mbox{for all $i\in I$.}
\end{align}
By definition the coproduct $\kow$ of $U_q(\gfrak')$ satisfies
\begin{align*}
  \kow(B_i) = B_i\ot K_i^{-1} + 1 \ot F_i - c_i K_{\tau(i)}K_i^{-1}\ot E_{\tau(i)} K_i^{-1}
\end{align*}
and hence $B_\bc$ is a right coideal subalgebra of $U_q(\gfrak')$, that is $\kow(B_\bc)\subseteq B_\bc\ot U_q(\gfrak')$. We call $B_\bc$ a quasi-split quantum symmetric pair coideal subalgebra of $U_q(\gfrak)$.
\begin{rema}
 The condition \eqref{eq:cond-c} on the parameters $\bc$ guarantees that the subalgebra $B_\bc$ has many desirable properties, see \cite[(5.9)]{a-Kolb14}, \cite[Proposition 3.1]{a-KY19p}.
\end{rema}  
\begin{rema}\label{rem:Satake}
  For $q$ not a root of unity, quantum symmetric pairs of Kac-Moody type were defined in \cite{a-Kolb14} depending on a pair $(X,\tau)$ where $\tau:I\rightarrow I$ is a diagram automorphism and $X$ is a subset of $I$ satisfying the admissibility conditions given in \cite[Definition 2.3]{a-Kolb14}. Following \cite{a-CLW18p} we call a quantum symmetric pair quasi-split if $X=\emptyset$. In the present paper we only consider quasi-split quantum symmetric pairs.

The definition of quantum symmetric pairs in \cite{a-Kolb14} involves a second parameter family $\bs=(s_i)_{i\in I}$. The corresponding coideal subalgebras $B_{\bc,\bs}$ are isomorphic as algebras for all $\bs$ under a map which maps generators to generators, see \cite[Theorem 7.1]{a-Kolb14}. In the present paper we are only concerned with the defining relations of $B_{\bc,\bs}$ and we hence restrict to the case $s_i=0$ for all $i\in I$.   
\end{rema}  
\subsection{The $\ast$-product on $H_\theta\ltimes U^-$}
We now recall a method devised in \cite{a-KY19p} to describe the algebra $B_\bc$ in terms of generators and relations. Let $U^-$ denote the subalgebra of $U_q(\gfrak')$ generated by all $F_i$ for $i\in I$. The algebra $U^-$ is $Q$-graded with $U^-_{-\mu}=\mathrm{span}_{\field}\{F_{i_1}\dots F_{i_m}\,|\,\sum_{j=1}^m\alpha_{i_j}=\mu\}$
for all $\mu\in Q^+=\N \Pi$, and $U^-_{-\mu}=\{0\}$ otherwise.
For any $i\in I$ let $\partial^R_i, \partial^L_i:U^-\rightarrow U^-$ denote the linear maps uniquely determined by the property that $\partial^R_i(F_j)=\partial_i^L(F_j)=\delta_{ij}$ for all $j\in I$ and
\begin{align}
   \partial^R_i(fg)&=q^{(\alpha_i,\nu)}\partial_i^R(f)g +  f\partial_i^R(g)\label{eq:partialR}\\
  \partial^L_i(fg)&=\partial_i^L(f)g + q^{(\alpha_i,\mu)} f\partial_i^L(g),\label{eq:partialL}
\end{align}
for all $f\in U^-_{-\mu}$, $g\in U^-_{-\nu}$. Consider the semidirect product $H_\theta\ltimes U^-$ which is the subalgebra of $U_q(\gfrak')$ generated by $H_\theta$ and $U^-$. The algebra $B_\bc$ is a deformation of $H_\theta\ltimes U^-$.
\begin{thm}\label{thm:BHU}
  (1) {\upshape\cite[Theorem 4.7, Lemma 5.2]{a-KY19p}} There exists an associative product $\ast$ on $H_\theta\ltimes U^-$ which is uniquely determined by the following properties:
      \begin{align}
        h\ast g&=hg, \qquad g\ast h=gh \qquad \mbox{for all $h\in H_\theta$, $g\in U^-$,}\label{eq:hastg}\\
        F_i\ast g&=F_ig - \frac{c_iq^{(\alpha_i,\alpha_{\tau(i)})}}{q_i-q_i^{-1}} K_{\tau(i)}K_i^{-1} \partial^L_{\tau(i)}(g) \qquad \mbox{for all $i\in  I$, $g\in U^-$.} \label{eq:Fastg}
      \end{align}
      (2)  {\upshape\cite[Corollary 5.8]{a-KY19p}} There is a uniquely determined isomorphism of algebras
      $$\psi:B_\bc\rightarrow (H_\theta\ltimes U^-,\ast)$$
such that $\psi(h)=h$ for all $h\in H_\theta$ and $\psi(B_i)=F_i$ for all $i\in I$.
\end{thm}
\begin{rema}
  Property \eqref{eq:Fastg} in Theorem \ref{thm:BHU} can be replaced by the property
  \begin{align}
     g\ast F_i&=gF_i - \frac{c_{\tau(i)}q^{(\alpha_i,\alpha_{\tau(i)})}}{q_i-q_i^{-1}}\partial^R_{\tau(i)}(g) K_iK_{\tau(i)}^{-1}  \qquad \mbox{for all $i\in  I$, $g\in U^-$.} \label{eq:gastF}
  \end{align}
  The resulting algebra structure on $H_\theta\ltimes U^-$ coincides with the algebra structure obtained in Theorem \ref{thm:BHU}.(1).
\end{rema}
\begin{rema}
  The coefficient in \eqref{eq:Fastg} differs from the corresponding coefficient in \cite[(4.25)]{a-KY19p}. This is due to the fact that we follow standard conventions \eqref{eq:EFFE} while \cite{a-KY19p} works with $E_iF_j-F_jE_i=\delta_{ij}(K_i-K_i^{-1})$. Moreover, our convention for the coefficient $c_i$ differs from \cite{a-KY19p} by a sign. The conventions in the present paper follow \cite{a-Kolb14} but we additionally allow $q$ to be a root of unity.   
\end{rema}  
Set $V^-=\bigoplus_{i\in I} \field F_i$ and let $T(V^-)$ denote the corresponding tensor algebra. By \cite{a-KY19p} the first part of the above theorem also holds when $U^-$ is replaced by $T(V^-)$. More precisely, there exists an associative product $\circledast$ on $H_\theta\ltimes T(V^-)$ which is uniquely determined by \eqref{eq:hastg} and \eqref{eq:Fastg} or \eqref{eq:gastF} for all $h\in H_\theta$, $g\in T(V^-)$, $i\in I$ with $\ast$ replaced by $\circledast$. By construction, the canonical projection gives rise to an algebra homomorphism
\begin{align*}
  \eta:(H_\theta\ltimes T(V^-),\circledast) \rightarrow (H_\theta\ltimes U^-,\ast)
\end{align*}
of deformed algebras. 
\begin{prop}\label{prop:eta-kernel}
  {\upshape\cite[Proposition 5.9]{a-KY19p}} The kernel of the algebra homomorphism $\eta$ is generated by the quantum Serre polynomials $S_{ij}(F_i,F_j)\in T(V^-)$ for $i,j\in I$.
\end{prop}  
Proposition \ref{prop:eta-kernel} and the second part of Theorem \ref{thm:BHU} together provide an effective method to obtain the defining relations for the algebra $B_\bc$. Indeed, the algebra $H_\theta\ltimes T(V^-)$ is generated over $H_\theta$ by the elements $F_i$ for $i\in I$ subject only to the relations $K_jK_{\tau(j)}^{-1} F_i = q^{-(\alpha_j-\alpha_{\tau(j)},\alpha_i)} F_i K_jK_{\tau(j)}^{-1}$. The additional relations in $B_\bc$ are obtained by rewriting the quantum Serre polynomials $S_{ij}(F_i,F_j)$ in terms of the deformed product $\circledast$ on $T(V^{-})$.

For any noncommutative polynomial $r(x_1,\dots,x_n)=\sum_J a_J \,x_{j_1}\dots x_{j_m}$ in $n$ variables with coefficients $a_J=a_{(j_1,\dots,j_m)}\in H_\theta$ and any elements $u_1,\dots,u_n\in H_\theta\ltimes T(V^-)$ we write 
\begin{align}\label{eq:*-poly}
  r(u_1 \stackrel{\circledast}{,} \dots  \stackrel{\circledast}{,} u_n ) = \sum_J a_J\, u_{j_1}\circledast \dots \circledast u_{j_m}.
\end{align}  
If $\tau(i)\neq \{i,j\}$ then \eqref{eq:Fastg} implies that $S_{ij}(F_i,F_j)=S_{ij}(F_i \stackrel{\circledast}{,} F_j)$. Hence it remains to consider the two cases $\tau(i)=i$ and $\tau(i)=j$.
\subsection{Deformed quantum Serre relations for $\tau(i)=i$}
All through this section we fix $i,j \in I$ with $\tau(i)=i\neq j$. In this case \eqref{eq:Fastg} and \eqref{eq:gastF} for $\circledast$ become
\begin{align}\label{eq:Fastg2}
  F_i\circledast g = F_ig + c \partial^L_i(g), \qquad g\circledast F= gF_i + c \partial^R_i(g)
\end{align}
where $c=-\frac{c_iq_i^2}{q_i-q_i^{-1}}$. For any polynomial $w(x,y)=\sum_{r,s}b_{rs} x^r y^s\in \field[x,y]$ and any $u_1, u_2,u_3\in H_\theta\ltimes T(V^-)$ set
\begin{align}\label{eq:curve-action}
  u_3 \curvearrowright w(u_1 \stackrel{\circledast}{,} u_2) = \sum_{r,s}b_{rs} u_1^{\circledast r} \circledast u_3 \circledast u_2^{\circledast s}.
\end{align}
\begin{lem}\label{lem:wmn}
  For any $m,n\in \N$ there exists a uniquely determined polynomial $w_{m,n}(x,y)=\sum_{r,s}b_{rs} x^r y^s\in \field[x,y]$ such that
  \begin{align*}
     F_i^m F_j F_i^n = F_j \curvearrowright w_{m,n}(F_i\stackrel{\circledast}{,} F_i).
  \end{align*}  
\end{lem}  
\begin{proof}
  By \eqref{eq:Fastg2} the noncommutative monomial $F_i^mF_j F_i^n$ can be written as a noncommutative polynomial with respect to the product $\circledast$ on $T(V^-)$. This polynomial is homogeneous of degree one in $F_j$ and hence can be written in the form $F_j \curvearrowright w_{m,n}(F_i\stackrel{\circledast}{,} F_i)$ for some polynomial $w_{m,n}(x,y)$ as in the lemma. The polynomial $w_{m,n}(x,y)$ is uniquely determined because the subalgebra of $(T(V^-),\circledast)$ generated by $F_i, F_j$ is a free algebra.  
\end{proof}
It remains to determine the polynomials $w_{m,n}(x,y)$ in the above Lemma. To this end observe that $\partial_i^L(F_i^n)=\partial^R_i(F_i^n)=(n)_{q_i^2}F_i^{n-1}$ where we use the non-symmetric quantum integer $(n)_p$ defined by $(n)_p=1+p+\dots+p^{n-1}$ for any $p\in \field$. Hence the first equation in \eqref{eq:Fastg2} and \eqref{eq:partialL} imply that
\begin{align*}
  F_i\circledast(F_i^m F_j F_i^n)= F_i^{m+1} F_j F_i^n +c (m)_{q_i^2} F_i^{m-1} F_j F_i^n + c q_i^{2m+a_{ij}} (n)_{q_i^2}F_i^m F_j F_i^{n-1}
\end{align*}
for $m,n\in \N\setminus \{0\}$. In view of Lemma \ref{lem:wmn} the above formula implies that the polynomials $w_{m,n}(x,y)$ satisfy the recursion
\begin{align}\label{eq:wmn-recursion}
  x w_{m,n}(x,y) = w_{m+1,n}(x,y) &+ c (m)_{q_i^2} w_{m-1,n}(x,y)\\
  &+ c q_i^{2m+a_{ij}} (n)_{q_i^2} w_{m,n-1}(x,y) \nonumber
\end{align}
for all $m,n\in \N\setminus \{0\}$. Similarly, using the second equation in \eqref{eq:Fastg2} and \eqref{eq:partialR} we obtain
\begin{align}
  \label{eq:wmn-recursion2}
  y w_{m,n}(x,y) = w_{m,n+1}(x,y) &+ c (n)_{q_i^2} w_{m,n-1}(x,y)\\
  &+ c q_i^{2n+a_{ij}} (m)_{q_i^2} w_{m-1,n}(x,y). \nonumber
\end{align}
The recursions \eqref{eq:wmn-recursion}, \eqref{eq:wmn-recursion2} also hold for $m=0$ or $n=0$ if we set $w_{-1,t}(x,y)=w_{s,-1}(x,y)=0$ for all $s,t\in \N$. Moreover, $w_{0,0}(x,y)=1$ as $F_j \curvearrowright 1=F_j$. The symmetry of the recursions \eqref{eq:wmn-recursion} and \eqref{eq:wmn-recursion2} implies that
\begin{align}\label{eq:wmn-symmetry}
  w_{m,n}(x,y)=w_{n,m}(y,x) \qquad \mbox{for all $m,n\in\N$.}
\end{align}  
Recall the bivariate $q$-Hermite polynomials $H_{m,n}(x,y;q,r)$ from Section \ref{sec:bivariate} which depend on two parameters $q,r$. The recursions \eqref{eq:wmn-recursion}, \eqref{eq:wmn-recursion2} imply that up to rescaling, $w_{m,n}(x,y)$ coincides with $H_{m,n}(x,y;q_i^2,q_i^{a_{ij}})$.
\begin{prop}\label{prop:wH}
  The polynomials $w_{m,n}(x,y)$ are given by
  \begin{align}\label{eq:wH}
    w_{m,n}(x,y)= \frac{H_{m,n}(b_ix,b_iy;q_i^2,q_i^{a_{ij}})}{(2b_i)^{m+n}}
  \end{align}
  where $b_i=\frac{1}{2}(q_i-q_i^{-1})c_i^{-1/2}q_i^{-1/2}$.
\end{prop}  
\begin{rema}
  The factor $b_i$ may lie in a quadratic extension of the field $\field$. However, the right hand side of \eqref{eq:wH} is still a well-defined polynomial in $\field[x,y]$ because if $x^iy^j$ appears in $H_{m,n}(x,y)$ with nonzero coefficient then $i+j\equiv m+n$ mod $2$.  
\end{rema}
\begin{proof}[Proof of Proposition \ref{prop:wH}]
  For a square root $b$ of a nonzero element in $\field$ and $m,n \in \N$ define a new polynomial $r_{m,n}(x,y)\in \field[x,y]$ by
  \begin{align*}
    r_{m,n}(x,y) = (2b)^{m+n} w_{m,n}(b^{-1}x,b^{-1}y).
  \end{align*}
  The recursion \eqref{eq:wmn-recursion} for $w_{m,n}$ is equivalent to
  \begin{align*}
    2x r_{m,n}(x,y)= r_{m+1,n}(x,y) &+ 4 c b^2 \frac{q_i^{2m}-1}{q_i^2-1}r_{m-1,n}(x,y)\\
    &+ 4cb^2  q_i^{2m+a_{ij}}\frac{q_i^{2n}-1}{q_i^2-1}r_{m,n-1}(x,y).
  \end{align*}
  Recall that $c=-\frac{c_iq_i^2}{q_i-q_i^{-1}}$ and hence the above recursion can be rewritten as
  \begin{align*}
    2xr_{m,n}(x,y) = r_{m+1,n}(x,y) &+ \frac{4c_iq_ib^2}{(q_i-q_i^{-1})^2}(1-q_i^{2m})r_{m-1,n}(x,y) \\
    &+ \frac{4c_i q_i b^2}{(q_i-q_i^{-1})^2} q_i^{a_{ij}} q_i^{2m}(1-q_i^{2n})r_{m,n-1}(x,y).
  \end{align*}
  For $b=\frac{1}{2}(q_i-q_i^{-1})c_i^{-1/2}q_i^{-1/2}$ the above recursion coincides with the recursion \eqref{eq:Hmn-recursion} for $H_{m,n}(x,y;q_i^2,q_i^{a_{ij}})$. Moreover, $r_{0,0}(x,y)=1$ and $r_{m,n}(x,y)=r_{n,m}(y,x)$ for all $m,n\in\N$ by \eqref{eq:wmn-symmetry}. Hence $r_{m,n}(x,y)=H_{m,n}(x,y;q_i^2,q_i^{a_{ij}})$ for this choice of $b$. 
\end{proof}
For any polynomial $w(x,y)=\sum_{r,s}b_{rs} x^r y^s\in \field[x,y]$ and any $u_1, u_2,u_3\in B_\bc$ set
\begin{align*}
  u \curvearrowright w(u_1, u_2) = \sum_{r,s}b_{rs} u_1^r\,  u_3\, u_2^s \in B_\bc,
\end{align*}
in analogy to the notation \eqref{eq:curve-action}. Combining Theorem \ref{thm:BHU}.(2), Proposition \ref{prop:eta-kernel}, Lemma \ref{lem:wmn} and Proposition \ref{prop:wH} we are now able to write down the deformed quantum Serre relations satisfied by the generators $B_i, B_j$ of $B_\bc$. Recall that in this section we always assume that $i=\tau(i)\neq j$.
\begin{cor}\label{cor:dqS-bivariate}
  The generators $B_i, B_j$ of $B_\bc$ satisfy the relation
  \begin{align}\label{eq:dqS-bivariate}
     \sum_{\ell=0}^{1-a_{ij}} (-1)^{\ell} \begin{bmatrix}1-a_{ij} \\ \ell \end{bmatrix}_{q_i}
B_j \curvearrowright w_{1-a_{ij} - \ell, \ell}(B_i, B_i) =0
  \end{align}
  where the polynomial $w_{m,n}(x,y)\in \field[x,y]$ is given by \eqref{eq:wH}.
\end{cor}
With $b_i$ as in Proposition \ref{prop:wH} we define univariate polynomials $w_m(x)\in\field[x]$ by
\begin{align}\label{eq:wm-qHermite}
  w_m(x)=w_{m,0}(x,y)=\frac{1}{(2b_i)^m} H_m(b_ix;q_i^2)
\end{align}
for all $m\in \N$. By \eqref{eq:wmn-recursion} the polynomials $w_{m}(x)$ satisfy the recursion
\begin{align}\label{eq:wm-recursion1}
  w_{m+1}(x)=x w_m(x) - c(m)_{q_i^2}w_{m-1}(x).
\end{align}

Note that the polynomials $w_m(x)$ depend on a choice of $i\in I$, but we do not make this explicit in the notation.
\begin{eg}\label{eg:wm}
For  small values of $m$ the polynomials $w_m(x)$ are given by
\begin{align*}
  w_0(x)&=1,\quad
  w_1(x)=x, \quad
  w_2(x)=x^2-c, \quad
  w_3(x)=x^3-(1+(2)_{q_i^2})cx,\\
  w_4(x)&= x^4 - c((1+(2)_{q^2_i}+(3)_{q^2_i})x^2+c^2(3)_{q_i^2}.
\end{align*}
\end{eg}
\begin{rema} 
   Define $w^{(n)}(x)\in \field[x]$ by
\begin{align*}
  w^{(n)}(x)= \frac{w_n(x)}{[n]_{q_i}^!} = \frac{H_n(b_i x;q_i^2)}{(2b_i)^n[n]_{q_i}^!}.
\end{align*}
  In terms of the divided powers $w^{(n)}(x)$ the recursion \eqref{eq:wm-recursion1} can be rewritten as
  \begin{align}\label{eq:wm-recursion2}
     [m]_{q_i}w^{(m)}(x) = x w^{(m-1)}(x) - c q_i^{m-2} w^{(m-2)}(x).
  \end{align}
  Interestingly, this recursion appeared for non quasi-split quantum symmetric pairs in \cite[(5.8)]{a-BW18p}.
\end{rema}  
Using Equations \eqref{eq:wH} and \eqref{bivariate in terms of univariate} we can express $w_{m,n}(x,y)$ in terms of the univariate polynomials $w_m(x)$. Additionally using the relations
\begin{align*}
(q^2;q^2)_k&=(-1)^k q^{k(k+1)/2}(q{-}q^{-1})^k[k]_q^!, & \frac{(q^2;q^2)_n}{(q^2;q^2)_m (q^2;q^2)_{n-m}}&= q^{m(n-m)}\begin{bmatrix}n \\ m \end{bmatrix}_{q}
\end{align*}
  for $n\ge m$, we obtain
\begin{align*}
  w_{m,n}(x,y) =\hspace{-.2cm} \sum_{k=0}^{\min(m,n)}  \hspace{-.2cm} (-1)^k c^k q_i^{k(m+n+a_{ij}-1)-\frac{k(k+1)}{2}} \begin{bmatrix}m \\ k \end{bmatrix}_{q_i}
  \hspace{-.1cm}\begin{bmatrix}n \\ k \end{bmatrix}_{q_i}\hspace{-.2cm}
    [k]_{q_i}^! w_{m-k}(x) w_{n-k}(y).
\end{align*} 
With this relation we calculate
\begin{align*}
  \sum_{n=0}^{1-a_{ij}}(-1)^n&\begin{bmatrix}1-a_{ij} \\ n \end{bmatrix}_{q_i} w_{1-a_{ij}-n,n}(x,y) = \sum_{n=0}^{1-a_{ij}}(-1)^n [1{-}a_{ij}]_{q_i}^! \cdot \\
  &\cdot \sum_{k=0}^{\min(1-a_{ij}-n,n)}\frac{(-1)^k c^k q_i^{-k(k+1)/2}}{[1{-}a_{ij}{-}n{-}k]_{q_i}^![n{-}k]_{q_i}^![k]_{q_i}^!} w_{1-a_{ij}-n-k}(x) w_{n-k}(y).
\end{align*}
Setting $\ell=n+k$ and $m=k$ we obtain
\begin{align}
  \sum_{n=0}^{1-a_{ij}}(-&1)^n\begin{bmatrix}1-a_{ij} \\ n \end{bmatrix}_{q_i} w_{1-a_{ij}-n,n}(x,y)\label{eq:qSerre-tauii}\\
  = \sum_{\ell=0}^{1-a_{ij}}(-1)^\ell
  \begin{bmatrix}1-a_{ij} \\ \ell \end{bmatrix}_{q_i}& w_{1-a_{ij}-\ell}(x)  \cdot\sum_{m=0}^{\lfloor\ell/2\rfloor}c^m q_i^{-m(m+1)/2}\begin{bmatrix}\ell \\ 2m \end{bmatrix}_{q_i} \frac{[2m]^!_{q_i}}{[m]^!_{q_i}} w_{\ell-2m}(y).\nonumber
\end{align}
In view of Equation \eqref{eq:qSerre-tauii} it is natural to consider a second family of polynomials $v_n(x)\in \field[x]$ defined for all $n\in \N$ by
\begin{align*}
  v_n(x) = \sum_{k=0}^{\lfloor n/2 \rfloor} c^k q_i^{-k(k+1)/2}\begin{bmatrix}n \\ 2k \end{bmatrix}_{q_i} \frac{[2k]^!_{q_i}}{[k]^!_{q_i}} w_{n-2k}(x)
\end{align*}
 The polynomials $v_n(x)$ can also be interpreted in terms of continuous $q$-Hermite polynomials. The proof of the first part of the following proposition is adapted from a similar calculation in the proof of \cite[Lemma 5.10]{a-BW18p}.
 \begin{prop}
   The polynomials $v_n(x)$ satisfy $v_{-1}(x)=0$, $v_0(x)=1$ and the recursion
   \begin{align}\label{eq:vm-recursion2}
       v_{m+1}(x) = x v_m(x) + c q_i^{-2}(m)_{q_i^{-2}} v_{m-1}(x) 
   \end{align}
   for all $m\in \N$. The polynomial $v_m(x)$ is given by
   \begin{align}\label{eq:vm-qHermite}
      v_m(x)=\frac{1}{(2b_i)^m} H_m(b_ix;q_i^{-2})  
   \end{align}
   where as before $b_i=\frac{1}{2}(q_i-q_i^{-1})c_i^{-1/2}q_i^{-1/2}$. 
 \end{prop}  
 \begin{proof}
   A direct calculation using \eqref{eq:wm-recursion1} gives 
   \begin{align*}
     &x v_m(x) + c q_i^{-2}(m)_{q_i^{-2}} v_{m-1}(x)\\
     &\stackrel{\eqref{eq:wm-recursion1}}{=}\sum_{k= 0}^{\lfloor m/2 \rfloor} c^k q_i^{-k(k+1)/2}
\begin{bmatrix}m \\ 2k \end{bmatrix}_{q_i} \frac{[2k]^!_{q_i}}{[k]^!_{q_i}}
  \bigg( w_{m-2k+1}(x) + c (m-2k)_{q_i^{2}} w_{m-2k-1}(x) \bigg)\\
  &\qquad \qquad+ c [m]_{q_i}q_i^{-m-1} \sum_{k= 0}^{\lfloor(m-1)/2\rfloor} c^k q_i^{-k(k+1)/2}\begin{bmatrix}m-1 \\ 2k \end{bmatrix}_{q_i} \frac{[2k]^!_{q_i}}{[k]^!_{q_i}}w_{m-1-2k}(x)\\
  &\stackrel{\phantom{\eqref{eq:wm-recursion1}}}{=} w_{m+1}(x)\\
  & + 
  \sum_{k= 1}^{\lfloor(m+1)/2 \rfloor}  \frac{c^k q_i^{-k(k+1)/2} [m]_{q_i}^!}{[m{+}1{-}2k]^!_{q_i}[k]_{q_i}^!}\bigg([m{+}1{-}2k]_{q_i} + [k]_{q_i}\big(q_i^{m+1-k}{+}q_i^{k-(m+1)})\bigg) w_{m+1-2k}(x)\\
     &\stackrel{\phantom{\eqref{eq:wm-recursion2}}}{=} v_{m+1}(x)
   \end{align*}
   which proves the recursion \eqref{eq:vm-recursion2}.
   Using $c=-\frac{c_iq_i^2}{q_i-q_i^{-1}}$ the recursion \eqref{eq:vm-recursion2} can be rewritten as
   \begin{align*}
     v_{m+1}(x)=x v_m(x) + \frac{c_iq_i}{(q_i-q_i^{-1})^2}(q_i^{-2m}-1)v_{m-1}(x).
   \end{align*}
   Now Equation \eqref{eq:vm-qHermite} follows by comparison with the recursion \eqref{q-hermite 3-term recursion}.   
 \end{proof}
 Combining Corollary \ref{cor:dqS-bivariate} with Equation \eqref{eq:qSerre-tauii} we are able to express the deformed quantum Serre relations in terms of univariate continuous $q$-Hermite polynomials.
 \begin{cor}\label{cor:dqS-univariate}
   Let $i,j\in I$ with $\tau(i)=i\neq j$. The generators $B_i,B_j$ of $B_\bc$ satisfy the relation
  \begin{align}\label{eq:qSerre-tauii2}
    \sum_{\ell=0}^{1-a_{ij}} (-1)^{\ell}\begin{bmatrix}1-a_{ij} \\ \ell \end{bmatrix}_{q_i} w_{1-a_{ij}-\ell}(B_i) B_j v_{\ell}(B_i)= 0.
  \end{align}
  where the polynomials $w_m(x)$ and $v_m(x)$ in $\field[x]$ are given by \eqref{eq:wm-qHermite} and \eqref{eq:vm-qHermite}, respectively.
 \end{cor}
 A resummation shows that \eqref{eq:qSerre-tauii2} can alternatively be written as
 \begin{align}\label{eq:qSerre-tauii3}
    \sum_{\ell=0}^{1-a_{ij}} (-1)^{\ell}\begin{bmatrix}1-a_{ij} \\ \ell \end{bmatrix}_{q_i} v_{1-a_{ij}-\ell}(B_i) B_j w_{\ell}(B_i)= 0.
 \end{align}
 \begin{rema}
   To make the relation between the polynomials $w_m(x)$ and $v_m(x)$ even clearer, we write $w_m(x;q_i,c)$ and $v_m(x;q_i,c)$ for the polynomials given by the recursions \eqref{eq:wm-recursion1} and \eqref{eq:vm-recursion2}, respectively, with $w_{-1}(x;q_i,c)=v_{-1}(x;q_i,c)=0$ and $w_0(x;q_i,c)=v_0(x;q_i,c)=1$.
 Then we have
 \begin{align} \label{eq:vmwm}
   v_m(x;q_i,c) = w_m(x;q_i^{-1},-q_i^{-2}c).
 \end{align}
 for all $m\in \N$.
 \end{rema}
\begin{rema}
  Corollary \ref{cor:dqS-univariate} can be used to calculate the deformed quantum Serre relation satisfied by the generators $B_i,B_j$ explicitly in the case $\tau(i)=i$ for small values of $-a_{ij}$. Using the expressions for $w_m(x)$ in Example \ref{eg:wm} and \eqref{eq:vmwm} one obtains
    \begin{align*}
      \sum_{n=0}^{1-a_{ij}}&(-1)^n\begin{bmatrix}1-a_{ij} \\ n \end{bmatrix}_{q_i} B_i^{1-a_{ij}-n}B_j B_i^n\\
      & = \begin{cases}
        \phantom{-}0 & \mbox{if $a_{ij}=0$,}\\
        -q_i c_i B_j & \mbox{if $a_{ij}=-1$,}\\
        -[2]_{q_i}^2 q_i c_i (B_i B_j - B_j B_i) & \mbox{if $a_{ij}=-2$,}\\
        -([3]_{q_i}^2+1)q_i c_i(B_i^2B_j+B_j B_i^2)& \\
        \qquad + [4]_{q_i} ([2]_{q_i}^2+1)q_ic_i B_iB_jB_i -[3]_{q_i}^2 (q_ic_i)^2 B_j & \mbox{if $a_{ij}=-3$.}
          \end{cases}
    \end{align*}
In slightly different conventions, these formulas first appeared in \cite[Lemma 2.2]{a-Letzter97}, \cite[Theorem 7.1]{a-Letzter03} for $q$ not a root of unity.    
\end{rema}
 \begin{rema}
  In \cite[Eq. (3.9)]{a-CLW18p} the relation \eqref{eq:qSerre-tauii2} is expressed in terms of so-called $\imath$divided powers for $q$ not a root of unity. Similarly to \eqref{eq:qSerre-tauii2} and \eqref{eq:qSerre-tauii3}, the $\imath$divided powers allow two equivalent expressions for the deformed quantum Serre relation in the case $\tau(i)=i\neq j$. It would be interesting to establish a relation between the $\imath$divided powers of \cite{a-CLW18p} and the continuous $q$-Hermite polynomials. 
 \end{rema}
 \begin{rema}
   Assume that $\field=k(q)$ is a field of rational functions in a variable $q$ over some field $k$ of characteristic zero. Let $\overline{\phantom{m}}:\field\rightarrow \field$ denote the bar involution sending a rational function $g(q)\in \field$ to $\overline{g(q)}=g(q^{-1})$. The map $\overline{\phantom{m}}$ extends to an involutive $k$-algebra automorphism $\overline{\phantom{m}}:\field[x]\rightarrow \field[x]$ by action on the coefficients. In this setting Equation \eqref{eq:vmwm} can be rewritten as
   \begin{align*}
      v_m(x) = \overline{w_m}(x) \qquad \mbox{if $\overline{c}=-q_i^{-2}c$.}
   \end{align*}  
  In the case $\overline{c}=-q_i^{-2}c$ the above formula and the equivalence of the relations \eqref{eq:qSerre-tauii2} and \eqref{eq:qSerre-tauii3} show that relation \eqref{eq:qSerre-tauii2} is preserved under the $k$-linear map given by $B_i\mapsto B_i$, $B_j\mapsto B_j$ and $q\mapsto q^{-1}$. This provides the essential step in the proof that the algebra $B_\bc$ has a bar-involution in the quasi-split case, as first observed in \cite[Proposition 3.7]{a-CLW18p}. Note that the existence of the bar-involution on $B_\bc$ also follows from the general theory in \cite{a-KY19p} without the need to have a presentation of $B_\bc$ in terms of generators and relations. This will be discussed elsewhere.  
 \end{rema}
\subsection{Deformed quantum Serre relations for $\tau(i)=j$}
The deformed quantum Serre relations for $\tau(i)=j$ were determined in \cite[Theorem 3.6]{a-BalaKolb15} based on Letzter's method \cite{a-Letzter03} involving coproducts. In this subsection we 
offer an alternative proof in the quasi-split case based on the star product method from \cite{a-KY19p}.
Throughout we fix distinct $i,j\in I$ with $\tau(i)=j$. In this case formula \eqref{eq:Fastg} for $\circledast$ becomes
\begin{align*}
  F_i \circledast g = F_ig + \gamma_i K_j K_i^{-1} \partial_j^L(g), \qquad
  F_j \circledast g = F_jg + \gamma_j K_i K_j^{-1} \partial_i^L(g)
\end{align*}
for all $g\in U^-$, where $\gamma_i=-\frac{c_iq_i^{a_{ij}}}{q_i-q_i^{-1}}$ and $\gamma_j=-\frac{c_j q_i^{a_{ij}}}{q_i-q_i^{-1}}$. Hence $F_i^n=F_i^{\circledast n}$ and
\begin{align}\label{eq:F*Fin}
  F_j \circledast F_i^n = F_j F_i^n + \gamma_j  q_i^{(n-1)(a_{ij}-2)} (n)_{q_i^2} F_i^{n-1} K_i K_j^{-1}.
\end{align}
By induction on $m$ one moreover gets
\begin{align*}
  F_i^{\circledast m}\circledast F_jF_i^n
   &=F_i^m F_j F_i^n + \gamma_i  q_i^{n(2-a_{ij})} (m)_{q_i^2} F_i^{m+n-1} K_j K_i^{-1}.
\end{align*}  
Inserting \eqref{eq:F*Fin} into the above equation we obtain
\begin{align}
  F_i^mF_jF_i^n= F_i^{\circledast m}\circledast F_j\circledast F_i^{\circledast n}
  &-\gamma_i  q_i^{n(2-a_{ij})} (m)_{q_i^2} F_i^{\circledast(m+n-1)} K_j K_i^{-1} \label{eq:Fim*FjFin}\\
  &- \gamma_j  q_i^{(n-1)(a_{ij}-2)} (n)_{q_i^2} F_i^{\circledast(m+n-1)} K_i K_j^{-1}. \nonumber
\end{align}
Using the relation
\begin{align}\label{eq:FimFjFin}
  \sum_{n=0}^\ell (-1)^n \begin{bmatrix}\ell \\ n \end{bmatrix}_{q} q^{n(\ell + 1)} =(q^2;q^2)_\ell \qquad \mbox{for all $\ell\in\N$}
\end{align}
one shows that
\begin{align}
   &\sum_{n=0}^{1-a_{ij}} (-1)^n \begin{bmatrix}1-a_{ij} \\ n \end{bmatrix}_{q_i} q_i^{n(2-a_{ij})}(1-a_{ij}-n)_{q_i^2} =-\frac{q_i^{-1}(q_i^{2};q_i^{2})_{1-a_{ij}}}{q_i-q_i^{-1}},\label{eq:sum1}\\
  &\sum_{n=0}^{1-a_{ij}} (-1)^n \begin{bmatrix}1-a_{ij} \\ n \end{bmatrix}_{q_i} q_i^{n(a_{ij}-2)}(n)_{q_i^2} =-\frac{q_i^{-1}(q_i^{-2};q_i^{-2})_{1-a_{ij}}}{q_i-q_i^{-1}}.\label{eq:sum2}
\end{align}
Recall the notation \eqref{eq:*-poly}. The formulas \eqref{eq:sum1}, \eqref{eq:sum2} and Equation \eqref{eq:Fim*FjFin} imply the relation  
\begin{align*}
  S_{ij}(F_i,F_j)=S_{ij}(F_i\stackrel{\circledast}{,}F_j) &+ \gamma_i \frac{q_i^{-1}(q_i^{2};q_i^{2})_{1-a_{ij}}}{q_i-q_i^{-1}}F_i^{\circledast(-a_{ij})} K_j K_i^{-1}\\
  &+ \gamma_j\frac{q_i^{1-a_{ij}}(q_i^{-2};q_i^{-2})_{1-a_{ij}}}{q_i-q_i^{-1}}F_i^{\circledast(-a_{ij})} K_i K_j^{-1}
\end{align*}
in $T(V^-)$. Using again Theorem \ref{thm:BHU}.(2) and Proposition \ref{prop:eta-kernel} one obtains the following result. 
\begin{thm}\label{thm:SBB2}
  Let $i,j\in I$ with $\tau(i)=j$ and $i\neq j$ and set $m=1-a_{ij}$. Then the relation
  \begin{align*}
    S_{ij}(B_i,B_j) = \frac{c_i q_i^{-m} (q_i^{2};q_i^{2})_{m}}{(q_i-q_i^{-1})^2} B_i^{m-1} K_j K_i^{-1}
    + \frac{c_jq_i (q_i^{-2};q_i^{-2})_{m}}{(q_i-q_i^{-1})^2}B_i^{m-1} K_i K_j^{-1}
  \end{align*}
  holds in $B_\bc$.
\end{thm}  

\end{document}